\documentclass[10pt,a4paper]{amsart}
\usepackage{amsmath,amsthm,amssymb,mathrsfs,hyperref}
\usepackage{aliascnt}
\usepackage{charter}
\DeclareMathAlphabet{\mathbsf}{OT1}{iwona}{b}{n}

\usepackage{xparse}
\usepackage{xcolor}
\definecolor{doom}{rgb}{0,0,0.68}
\hypersetup{
	unicode=true,
	colorlinks=true,
	citecolor=doom,
	linkcolor=doom,
	anchorcolor=doom
}

\makeatletter
\expandafter\g@addto@macro\csname th@plain\endcsname{%
		\thm@notefont{\bfseries}
	}%
\makeatother

\newtheorem{theorem}{Theorem}
\newtheorem*{theorem*}{Theorem}
\newtheorem*{lemma*}{Lemma}

\newaliascnt{lemma}{theorem}
\newaliascnt{prop}{theorem}
\newaliascnt{cor}{theorem}
\newaliascnt{conj}{theorem}
\newaliascnt{fact}{theorem}
\newaliascnt{example}{theorem}
\newaliascnt{defn}{theorem}

\newtheorem{lemma}[lemma]{Lemma}
\newtheorem{prop}[prop]{Proposition}

\newtheorem{fact}[fact]{Fact}

\aliascntresetthe{lemma}
\aliascntresetthe{prop}
\aliascntresetthe{cor}
\aliascntresetthe{conj}
\aliascntresetthe{fact}
\aliascntresetthe{example}

\newaliascnt{claim}{theorem}

\aliascntresetthe{claim}

\theoremstyle{definition}

\newtheorem*{definition*}{Definition}

\aliascntresetthe{defn}

\newcommand{\axiom}[1]{\mathbsf{#1}}  % for axioms
\newcommand{\ZFA}{\axiom{ZFA}}

\newcommand{\ZFC}{\axiom{ZFC}}
\newcommand{\AC}{\axiom{AC}}
\newcommand{\DC}{\axiom{DC}}
\newcommand{\ZF}{\axiom{ZF}}
\newcommand{\HS}{\axiom{HS}}

\newcommand{\Ord}{\axiom{Ord}}

\newcommand{\WISC}{\axiom{WISC}}

\newcommand{\cG}{\mathscr G}
\newcommand{\cF}{\mathcal F}
\newcommand{\fM}{\mathfrak M}
\newcommand{\fN}{\mathfrak N}

\newcommand{\PP}{\mathbb P}

\newcommand{\cZ}{\mathcal Z}

\newcommand{\DD}{\mathbb D}

\DeclareMathOperator{\id}{id}
\DeclareMathOperator{\sym}{sym}
\DeclareMathOperator{\fix}{fix}
\DeclareMathOperator{\dom}{dom}
\DeclareMathOperator{\rng}{rng}
\DeclareMathOperator{\aut}{Aut}

\mathchardef\mhyphen="2D

\newcommand{\forces}{\Vdash}

\newcommand{\power}{\mathcal P}
\newcommand{\res}{\!\!\upharpoonright}

\newcommand{\midd}{\mathrel{}\middle|\mathrel{}}
\NewDocumentCommand \set{mo}{
		\IfNoValueTF{#2}
			{\left\{ #1 \right\}}
			{\left\{ #1 \midd #2 \right\}}}

\newcommand{\expr}[1]{\ulcorner #1 \urcorner}
\address{\textbf{Einstein Institute of Mathematics}\\
Edmond J. Safra Campus, Givat Ram\\
The Hebrew University of Jerusalem.\\
Jerusalem, 91904, Israel}
\title{Embedding Orders Into the Cardinals With $\DC_\kappa$}
\author{Asaf Karagila}
\email{karagila@math.huji.ac.il}
\urladdr{http://boolesrings.org/asafk}
\date{May 31, 2013.}

\begin{document}
%\setpagewiselinenumbers
%\modulolinenumbers[1]
%\linenumbers

\begin{abstract}
Jech proved that every partially ordered set can be embedded into the cardinals of some model of $\ZF$. We extend this result to show that every partially ordered set can be embedded into the cardinals of some model of $\ZF+\DC_{<\kappa}$ for any regular $\kappa$. We use this theorem to show that for all $\kappa$, the assumption of $\DC_\kappa$ does not entail that there are no decreasing chains of cardinals. We also show how to extend the result to and embed into the cardinals a proper class which is definable over the ground model. We use this extension to give a large cardinals-free proof of independence of the weak choice principle known as $\WISC$ from $\DC_\kappa$.
\end{abstract}
\maketitle

\section{Introduction}

Assuming the axiom of choice, cardinals trivially form a well-ordered class, but with its failure their order structure can be as complex as desired. An interesting example for this range of possibilities is Jech's theorem in which he proves that if $(P,\leq)$ is a partial order then there exists a model of $\ZFA$ (Zermelo-Fraenkel with Atoms) in which $(P,\leq)$ can be embedded into the cardinals (see \cite{Jech:ORD1966}). The theorem was complemented by the Jech-Sochor embedding theorem which allowed to carry the consistency result into $\ZF$ and remove the need for atoms (see \cite{JechSochor:1966a, JechSochor:1966b} and \cite[Chapter~6]{Jech:AC1973}). This theorem tells us, essentially, that there are no limitations on the order structure of cardinals defined by injections.

Jech's original proof included adding many counterexamples to $\DC$ (in the form of Dedekind-finite sets). While it can be modified to allow $\DC_\kappa$ to hold, the Jech-Sochor theorem is not suited to transfer universal statements such as $\DC_\kappa$. We should point out that Pincus improved upon Jech-Sochor's original work and showed that it is possible to transfer injectively boundable statements (a class of statements which include $\DC_\kappa$), and more. In this paper we give a direct forcing argument to Jech's proof, and this allows us to preserve $\DC_\kappa$ up to an arbitrary (but fixed) cardinal $\kappa$. We then proceed to show that in fact ordered classes (which are definable in the ground model) may be embedded into the cardinals while preserving $\DC_\kappa$.

The authors of \cite{MooreBanas:1990} remark that it is unknown whether or not the assumption that there are no decreasing sequences of cardinals (in the $\leq$ relation) implies that the axiom of choice holds. They write in section four: ``The answer [...] is almost certainly negative, but thus far there is no proof.'' and mention that there has been some disagreement on the topic in the past. We will use the improved embedding theorem (of partial orders into the cardinals) to show that for any $\kappa$ it is consistent with $\ZF+\DC_\kappa$ that there are decreasing chains of cardinals. While there is no positive results yet, this somehow suggests that the axiom of choice might be equivalent to the assertion ``There is no infinite decreasing sequence of cardinals''.

Decreasing sequences of cardinals also have importance in answering a question of Feldman \& Orhon which appeared in \cite{FeldmanOrhonBlass:2008}. In the paper the authors prove \footnote{The original proof is due to Tarski, see \cite[Form~T3(n), pp. 22-23]{RubinRubinHoward:1985}.} that for any $k\in\omega\setminus\{0,1\}$ the assertion that ``Every antichain of cardinals has less than $k$ members'' implies the axiom of choice, and a question asks about replacing the finite bound by $\omega$. Feldman and Orhon conjectured that ``Every antichain of cardinals is finite'' does not imply the axiom of choice in $\ZF$. The question is still open, but we will show that $\ZF+\DC_\kappa$ cannot prove that every antichain is finite.

In \cite{Roguski:1990} the author proves that it is consistent relative to the consistency of $\ZF$ that for every set of cardinals there exists one incomparable with all of them. We extend this result and show its compatibility with $\DC_\kappa$. We use this extension to show the consistency of long chains and antichains of cardinals, and to give a large cardinals-free consistency result of the failure of $\WISC$, a recent choice principle related to constructive set theory.

\subsection{Clarification}
After the acceptance and revision of the paper it was pointed out to the author that Takahashi proved in \cite{Tahakashi:1968} results in the vein of some of the results in this paper. He shows that a partial order can be embedded into subsets of the continuum, a result which is generalized in Section~3 of this paper. Takahashi infers the existence of a decreasing chain of cardinals of order $\omega^*$, as we show possible in Section~5. In this paper we extend both results to a much broader context, and the proofs presented here are written in a modern format using unramified forcing and symmetric models.

\section{Basic Definitions}
Suppose that $\fM$ is a countable transitive model of $\ZFC$, a notion of forcing $\PP=(P,\leq)\in\fM$ is a partial order with a maximum denoted by $1_\PP$. The elements of $P$ are called conditions and when $p\leq q$ we say that $p$ extends $q$, or that $p$ is stronger than $q$. We say that $p$ and $q$ are \textbf{compatible} if there is $r$ which extends both of them, otherwise $p$ and $q$ are \textbf{incompatible}. If a certain definition of $\PP$ does not result in a partial order that has a maximum, we add one artificially. We will also consider only non-trivial notions of forcing, that is to say that every $p\in P$ has two incompatible extensions.

We define by induction the class of $\PP$-names (calculated within $\fM$):
\begin{enumerate}
\item $\fM^\PP_0 = \varnothing$;
\item $\fM^\PP_{\alpha+1} = \power(P\times\fM^\PP_\alpha)$;
\item $\fM^\PP_\delta = \bigcup_{\alpha<\delta}\fM^\PP_\alpha$ for a limit ordinal $\delta$.
\end{enumerate} 
Finally the class of $\PP$-names is \[\fM^\PP = \bigcup_{\alpha\in\Ord}\fM^\PP_\alpha,\] 
where $\power(x)$ denotes the power set of $x$, and $\Ord$ denotes the class of ordinals. We will use $\dot x$ to denote a $\PP$-name, and $\check x$ to denote a canonical name for $x\in\fM$. If $G$ is a $\PP$-generic filter over $\fM$ then $\dot x^G$ is the interpretation of $\dot x$ by the filter $G$.

Let $\set{\dot x_i}[i\in I]$ be a class of $\PP$-names in $\fM$ (if it is a proper class we then require it to be definable). We denote by $\set{\dot x_i}[i\in I]^\bullet$ the name $\set{(1_\PP,\dot x_i)}[i\in I]$. We shall also use $(\dot x,\dot y)^\bullet$ to denote the canonical name for the ordered pair, namely $\set{\set{\dot x}^\bullet,\set{\dot x,\dot y}^\bullet}^\bullet$.

Suppose that $\pi$ is an automorphism of $\PP$, we may extend $\pi$ as an automorphism of $\PP$-names by induction, \[\widetilde{\pi}(\dot x)=\set{(\pi p,\widetilde\pi\dot y)}[(p,\dot y)\in\dot x].\]
From this point, though, we will only use $\pi$ to denote the automorphism of $\PP$ as well the automorphism of the $\PP$-names. If $\PP$ was defined using parameters from $A$ then a permutation of $A$ can be used to define an automorphism of $\PP$. This will be the case in our proofs. It can be shown by induction that if $x\in\fM$ then $\pi\check x=\check x$ for any $\pi\in\aut(\PP)$.

\begin{lemma*}[The Symmetry Lemma]\label{Lemma:SymmetryLemma}
Let $\varphi(u_1,\ldots, u_n)$ be a formula in the language of set theory, $p$ a condition in $\PP$ and $\dot x_1,\ldots,\dot x_n$ are $\PP$-names, and $\pi\in\aut(\PP)$. Then \[p\forces\varphi(\dot x_1,\ldots,\dot x_n)\iff\pi p\forces\varphi(\pi\dot x_1,\ldots,\pi\dot x_n)\]
\end{lemma*}

The proof is by induction on the complexity of $\varphi$, and can be found in \cite[Lemma~7.13(c)]{Kunen:1980}. 

Suppose that $\cG$ is a group of permutations of a set $A$, and $E\subseteq A$; we define pointwise stabilizer of $E$ as the group $\fix_\cG(E)=\set{\pi\in\cG}[\pi\res E=\id_E]$. If $\cG$ acts on $\PP$-names (through its action on $\PP$ in most cases) we define the stabilizer of the name $\dot x$ as the group $\sym_\cG(\dot x)=\set{\pi\in\cG}[\pi\dot x=\dot x]$. We omit $\cG$ from these notations if it is clear from context.

If $\cG$ is a group, we say that $\cF\subseteq\power(\cG)$ is \textbf{a filter of subgroups} if whenever $H\in\cF$ and $H\leq K$ then $K\in\cF$, and $\cF$ is closed under finite intersections. We also require the trivial group is not in $\cF$. We say that $\cF$ is \textbf{normal} if it is closed under conjugation.

If $\cG$ is a group of permutations of $\PP$ (or acting on it), and $\cF$ is a normal filter of subgroups of $\cG$ we say that $\dot x\in\fM^\PP$ is an \textbf{$\cF$-symmetric} name if $\sym(\dot x)\in\cF$. We define the class of \textbf{hereditarily $\cF$-symmetric sets} by induction, $\dot x$ is hereditarily $\cF$-symmetric if and only if $\dot x$ is $\cF$-symmetric, and for every $(p,\dot y)\in\dot x$, $\dot y$ is hereditarily $\cF$-symmetric. We shall denote by $\HS_\cF$ the class of hereditarily $\cF$-symmetric names, and as usual we will omit $\cF$ when it is clear from the context.

Let $G$ be a $\PP$-generic filter over $\fM$, and denote by $\fN$ the class $(\HS_\cF)^G=\set{\dot x^G}[\dot x\in\HS_\cF]$. $\fN$ is called a \textbf{symmetric extension} (generated by $\cF$) of $\fM$. The following theorem is stated, and its proof can be found in \cite{Jech:ST2003}.

\begin{theorem*}
$\fN$ is a transitive model of $\ZF$ and $\fM\subseteq\fN\subseteq\fM[G]$.\qed
\end{theorem*}

For a cardinal $\kappa$ we denote by $\DC_\kappa$ the \textbf{Principle of Dependent Choice for $\kappa$} which states that for every non-empty set $X$, if $R$ is a binary relation such that for every ordinal $\alpha<\kappa$, and every $f\colon\alpha\to X$ there is some $y\in X$ such that $f\mathrel{R} y$ then there is $f\colon\kappa\to X$ such that for every $\alpha<\kappa$, $f\res\alpha\mathrel{R}f(\alpha)$. We shall abbreviate by $\DC_{<\kappa}$ the assertion $(\forall\lambda<\kappa)\DC_\lambda$.

The axiom of choice implies that $\DC_\kappa$ holds for every $\kappa$, and in fact $\forall\kappa.\DC_\kappa$ is equivalent to the axiom of choice. One useful consequence of $\DC_\kappa$ is that for every set $X$ there is either an injection from $X$ into $\kappa$ or an injection from $\kappa$ into $X$. One can find a thorough treatment of $\DC_\kappa$ and related choice principles in \cite[Chapter~8]{Jech:AC1973}.

\begin{lemma}\label{Lemma:ClosureDC}
Let $\cF$ be a normal filter of subgroups of a group of automorphisms of $\PP$, and $\fN=(\HS_\cF)^G$ be the symmetric extension of $\fM$. If $\PP$ is $\kappa$-closed and $\cF$ is a $\kappa$-complete filter then $\fN\models\DC_{<\kappa}$.
\end{lemma}
\begin{proof}
Let $\lambda<\kappa$, we will show that if $f\colon\lambda\to\fN$ is in $\fM[G]$ then $f\in\fN$. From this it follows that $\fN\models\DC_\lambda$, because if $X$ and $R$ are elements of $\fN$ as in the assumptions of $\DC_\lambda$, then we can find $f\colon\lambda\to X$ in $\fM[G]$ (as the latter is a model of $\AC$), and by the proof here we will have that $f\in\fN$.

Let $\dot f_0$ be a name for $f$ and let $p$ be a condition forcing that $\dot f_0$ is a function whose domain is $\lambda$ and its range is a subset of $\fN$. Because $\PP$ is $\kappa$-closed we can extend $p$ to $p_0\geq p_1\geq\dots\geq p_\alpha\geq\dots\geq p_\lambda$ such that for all $\alpha<\lambda$, $p_\alpha\forces\dot f_0(\check\alpha)=\dot t_\alpha$, where $\dot t_\alpha\in\HS$. Then we can define the collection $\set{\dot t_\alpha}[\alpha<\lambda]$ in $\fM$, and take $\dot f=\set{\dot t_\alpha}[\alpha<\lambda]^\bullet$. Clearly $\dot f^G=f$ whenever $p_\lambda\in G$. We need to show that $\dot f\in\HS$, but it is enough to show that $\sym(\dot f)\in\cF$ because all the names appearing in $\dot f$ are taken from $\HS$ to begin with. We have that for every $\pi\in\bigcap_{\alpha<\lambda}\sym(\dot t_\alpha)$ it holds that $\pi\dot f=\dot f$, and by $\kappa$-completeness of $\cF$ we have that the intersection is in $\cF$, and so $\dot f$ is in $\HS$ as wanted.
\end{proof}

Remember that if $A$ is a set the $|A|$ is the \textbf{cardinal number} of $A$. While in $\ZFC$ cardinal numbers are all ordinals, without the axiom of choice it is not always the case. We define $|A|$ to be the least ordinal bijectible with $A$ if such ordinal exists, and otherwise $|A|$ is the set of those $B$ which are in bijection with $A$ and have minimal rank with respect to this property. If $|A|$ is a finite ordinal we say that $A$ is finite, if it is an infinite ordinal we say that $|A|$ is an \textbf{aleph number}; in both cases we may say that $|A|$ is a well-ordered cardinal. If $|A|$ is not a well-ordered cardinal we say that $A$ is \textbf{not well-orderable}.

For $A,B$ sets we define $|A|\leq|B|$ if and only if there is an injection from $A$ into $B$, and $|A|\leq^\ast|B|$ if and only if $A$ is empty or there is a surjection from $B$ onto $A$. Both relations are reflexive and transitive, but only $\leq$ is provably antisymmetric without the axiom of choice. We also have that $|A|\leq|B|$ implies $|A|\leq^\ast|B|$. For further analysis of the $\leq^\ast$ relation see \cite{MooreBanas:1990}.

\section{Embedding Partially Ordered Sets Into Cardinals}

Let $\fM$ be a countable transitive model of $\ZFC$, $\kappa$ regular in $\fM$. Let $(\cZ,\leq)\in\fM$ be a partially ordered set. We want to embed $(\cZ,\leq)$ into the cardinals of some model, but instead we will embed $(\power(\cZ),\subseteq)$. We observe that $(\cZ,\leq)$ itself embeds into its power set by the map $z\mapsto\set{z'\in\cZ}[z'\leq z]$, and so it is indeed enough to embed the power set of $\cZ$. 

We define $\PP=(P,\leq)$ to be the following notion of forcing defined within $\fM$. $p\in P$ is a partial function $p\colon(\cZ\times\kappa)\times\kappa\to 2$ such that $|\dom p|<\kappa$. As usual $p\leq q\iff q\subseteq p$. We note that this forcing is $\kappa$-closed and therefore does not collapse cardinals smaller than $\kappa^+$. If $\kappa^{<\kappa}=\kappa$, then $\PP$ has $\kappa^+$-c.c. and no cardinals are collapsed.

If $G$ is $\PP$-generic over $\fM$ then $\bigcup G=g$ is a total function from $(\cZ\times\kappa)\times\kappa$ to $2$ in $\fM[G]$. We define the following generic sets and we give them canonical names:
\begin{itemize}
\item Let $z\in\cZ,\alpha\in\kappa$ we define $r_{z,\alpha}=\set{\gamma<\kappa}[g((z,\alpha),\gamma)=1]$, with the canonical name $\dot r_{z,\alpha}=\set{(p,\check\gamma)}[p((z,\alpha),\gamma)=1]$.
\item Let $z\in\cZ$ we define $R_z=\set{r_{z,\alpha}}[\alpha<\kappa]$ with the canonical name $\dot R_z=\set{\dot r_{z,\alpha}}[\alpha<\kappa]^\bullet$.
\item Let $Q\subseteq\cZ$ we define $D_Q=\bigcup_{z\in Q} R_z$. We do not give a canonical name to $D_Q$, because we allow $Q\notin\fM$.
\end{itemize}

Let $\cG$ the group of all permutations of $\cZ\times\kappa$ such that for all $(z,\alpha)$ we have $\pi(z,\alpha)=(z,\beta)$ for some $\beta$ (note that $\pi(z_1,\alpha_1)=(z_1,\beta)$ and $\pi(z_2,\alpha_2)=(z_2,\beta)$ does not imply $\alpha_1=\alpha_2$). We define the action of $\cG$ on $\PP$. If $\pi\in\cG$ we define \[\pi p(\pi(z,\alpha),\gamma)=p((z,\alpha),\gamma).\]

We extend the action of $\cG$ to the class of $\PP$-names. We first make the following observation. For all $z\in\cZ,\alpha<\kappa$ we have that \[\pi\dot r_{z,\alpha}=\set{(\pi p,\pi\check\gamma)}[p((z,\alpha),\gamma)=1]=\set{(p,\check\gamma)}[p(\pi(z,\alpha),\gamma)=1]=\dot r_{\pi(z,\alpha)},\] and for any $z\in\cZ$ we have that $\pi\dot R_z=\dot R_z$.

Let $I=[\cZ\times\kappa]^{<\kappa}$, let $\cF$ be the filter generated by $\fix(E)$ for $E\in I$, namely \[\cF=\set{H\leq\cG}[\exists E\in I:\fix(E)\leq H].\] It is left to the reader to verify that $\cF$ is indeed a normal filter of subgroups and that $\cF$ is $\kappa$-closed. Let $\HS$ denote the class of hereditarily symmetric $\PP$-names, and let $\fN$ be $\HS^G$. If $E\in I$ and $\fix(E)\leq\sym(\dot x)$ we say that $E$ is a \textbf{support} of $\dot x$.

We have that $\fN\subseteq\fM[G]$ is a model of $\ZF$. It follows from the $\kappa$-closure of $\PP$ and $\cF$ that the conditions for \autoref{Lemma:ClosureDC} hold, and thus $\fN\models\DC_{<\kappa}$. We will see later that the axiom of choice, indeed $\DC_\kappa$ itself does not hold in $\fN$.

\begin{prop}\label{p:canonical-symmetric}
For all $z\in\cZ$ and $\alpha<\kappa$, $r_{z,\alpha}\in\fN$ and $R_z\in\fN$.
\end{prop}
\begin{proof}
The above observation shows that for every $\pi\in\cG$ and $(z,\alpha)\in\cZ\times\kappa$,\[\pi\dot r_{z,\alpha}=\dot r_{\pi(z,\alpha)},\ \pi\dot R_z=\dot R_z.\]
From that follows that $\set{(z,\alpha)}$ is a support for $\dot r_{z,\alpha}$ (and clearly every name appearing in $\dot r_{z,\alpha}$ is symmetric, being a canonical name of an ordinal). Therefore $\dot r_{z,\alpha}\in\HS$. Now all the names appearing in $\dot R_z$ are from $\HS$, and having $\varnothing$ as a support we have that $\dot R_z\in\HS$ as well. Therefore the sets $r_{z,\alpha}$, $R_z$ are all in $\fN$.
\end{proof}

Two facts which are useful for later are:
\begin{fact}
For every $Q\in\fN$ such that $Q\subseteq\cZ$, $D_Q\in\fN$. Moreover, the function $Q\mapsto D_Q$ is in $\fN$.
\end{fact}
\begin{proof}
Note that $\dot F=\set{\left(\check z,\dot R_z\right)^\bullet}[z\in\cZ]^\bullet$ is in $\HS$, since for every $z$ we have that $\sym(\dot R_z)=\cG$. Let $F=\dot F^G$, then for $Q\subseteq\cZ$ we have $D_Q=\bigcup_{z\in Q} F(z)$. Therefore whenever $Q\in\fN$, $D_Q\in\fN$ as well.
\end{proof}

We remark that from a name $\dot Q\in\HS$ for a subset of $\cZ$ one can give a (relatively) canonical name for $D_Q$ which has the same support as $\dot Q$. However by showing that $D_Q\in\fN$ is definable from $Q\in\fN$ we in fact prove that there is such name.

\begin{fact}\label{Fact:Onto}
The following is true in $\fN$. For every $z\in\cZ$ we have that $R_z$ can be mapped onto $\kappa$, and therefore $D_Q$ can be mapped onto $\kappa$ for every non-empty $Q$.
\end{fact}
\begin{proof}
The map $r_{z,\alpha}\mapsto\min r_{z,\alpha}$ is well-defined in $\fN$, and by a simple density argument one can see it is surjective in $\fM[G]$ and therefore in $\fN$ as well. 
\end{proof}

\begin{prop}\label{Prop:ACfails}
In $\fN$ it is true that for every $z\in\cZ$ we have that $R_z$ cannot be well-ordered, and therefore the axiom of choice fails.
\end{prop}
\begin{proof}
We know that $R_z$ can be mapped onto $\kappa$, therefore it would suffice to show that there is no injection from $\kappa$ into $R_z$. If $R_z$ could have been well-ordered such a surjection could have been reversed to an injection from $\kappa$.

Towards a contradiction that $p\forces\expr{\dot f\colon\check\kappa\to\dot R_z\text{ is injective}}$, and $\dot f\in\HS$. Let $E$ be a support for $\dot f$, and let $q\leq p$ be a condition that there are $\alpha,\tau<\kappa$ such that $(z,\alpha)\notin E$ and $q\forces\dot f(\check\tau)=\dot r_{z,\alpha}$.

We can now find $\beta\neq\alpha$ such that $(z,\beta)\notin E$ and for all $\gamma$, $((z,\beta),\gamma)\notin\dom q$. We define the following $\pi\in\cG$: $\pi(z,\alpha)=(z,\beta), \pi(z,\beta)=(z,\alpha)$ and $\pi(x,y)=(x,y)$ otherwise. Clearly $\pi\in\fix(E)$ and therefore $\pi\dot f=\dot f$. By the symmetry lemma we have that \[\pi q\forces\dot f(\check\tau)=\dot r_{z,\beta}.\]

If $q$ and $\pi q$ are compatible then $q$ has an extension which forces that $\dot f$ is not a function, which is a contradiction. Suppose $((t,\varepsilon),\delta)\in\dom q\cap\dom\pi q$, if $t\neq z$ then $\pi(t,\varepsilon)=(t,\varepsilon)$ and by the definition of $\pi q$ we have \[\pi q((t,\varepsilon),\delta)=\pi q(\pi(t,\varepsilon),\delta)=q(t,\varepsilon,\delta).\] Otherwise $t=z$, if $\varepsilon\notin\set{\alpha,\beta}$ then $\pi(t,\varepsilon)=(t,\varepsilon)$ and so $q((t,\varepsilon),\delta)=\pi q((t,\varepsilon),\delta)$. Moreover, if $t=z$, then $\varepsilon\neq\beta$. Recall the choice of $\beta$ was such that: \[((z,\beta),\delta)\notin\dom q.\] Finally, if $((z,\alpha),\delta)\in\dom\pi q$, then $(\pi^{-1}(z,\alpha),\delta)=((z,\beta),\delta)\in\dom q$, and so it is impossible that $t=z$ and $\varepsilon=\alpha$. Therefore $q$ and $\pi q$ agree on all the points in their common domain, and are compatible, which is our desired contradiction.

Therefore in $\fN$ there is no injection from $\kappa$ into $R_z$, and choice fails.
\end{proof}

We have in fact shown that $\kappa$ and $R_z$ have incomparable cardinalities in $\fN$, and therefore $\DC_\kappa$ fails as promised.

\begin{theorem}\label{Thm:IncomparabilityOfChunks}
Let $Q,T\subseteq\cZ$ in $\fN$. If $Q\nsubseteq T$ then $\fN\models |D_Q|\nleq|D_T|$ and $|D_Q|\nleq^\ast|D_T|$.
\end{theorem}
\begin{proof}
If there had been an injection from $D_Q$ into $D_T$ then there would have been a surjection from $D_T$ onto $D_Q$. It is therefore sufficient to argue for the $\nleq^\ast$ case.

Let $\dot Q$ and $\dot T$ be names for $Q$ and $T$ respectively, both in $\HS$, and let $\dot D_Q$ and $\dot D_T$ be names in $\HS$ for $D_Q$ and $D_T$ respectively.

Suppose that $p\forces\expr{\dot f\colon\dot D_T\to\dot D_Q\text{ is surjective}}$ for some $\dot f\in\HS$, we will prove that $p\forces\dot Q\subseteq\dot T$. Assume towards contradiction that is not the case, if $p$ does not decide the statement $\dot Q\nsubseteq\dot T$, then it has an extension deciding it and we shall take it instead. So we may assume a stronger assumption towards contradiction, $p\forces\dot Q\nsubseteq\dot T$. Let $E$ be a support for the names $\dot f,\dot Q,\dot T,\dot D_Q,\dot D_S$. 

Let $q\leq p$ be an extension such that there are some $z,t\in\cZ$ and $\alpha,\delta<\kappa$ such that $(z,\alpha)\notin E$ and $q\forces\check z\in\dot Q\setminus\dot T\land\check t\in\dot T\land\dot f(\dot r_{t,\delta})=\dot r_{z,\alpha}$. This implies that $z\neq t$. Let $\beta\neq\alpha$ be such that $(z,\beta)\notin E$ and there is no $\gamma<\kappa$ for which $((z,\beta),\gamma)\in\dom q$. We define $\pi$ to be the permutation such that $\pi(z,\alpha)=(z,\beta),\pi(z,\beta)=(z,\alpha)$ and $\pi(x,y)=(x,y)$ otherwise. We have that $\pi\in\fix(E)$ and therefore all the names of interest are not changed by $\pi$.

We have that $\pi q\forces\dot f(\dot r_{t,\delta})=\dot r_{z,\beta}$. However a simple  verification as in the proof of \autoref{Prop:ACfails} shows that $q$ and $\pi q$ are compatible and therefore $q$ has an extension which forces $\dot f$ is not a function, which is a contradiction.
\end{proof}

\section{Embedding a Proper Class}

In this section we extend the result by Roguski (\cite{Roguski:1990}) in which he proves the following theorem:

\begin{theorem*}[Roguski] Let $\fM$ be a countable transitive model of $\ZFC$, and $(I,\preceq)$ a partially ordered class such that $I,\preceq$ are both classes of $\fM$ and every initial segment of $(I,\preceq)$ belongs to $\fM$. Then there is a countable transitive model $\fN$ for $\ZF$, which is a symmetric extension of $\fM$ and a class $\set{S_i}[i\in I]$ in $\fN$ such that for all $i,j\in I$, \[i\preceq j \leftrightarrow\fN\models|S_i|\leq|S_j|.\]
\end{theorem*}

From this theorem he draws the consistency of a proper class of pairwise incomparable cardinals. However, it seems that Roguski is proving less than he claims to prove. Roguski embeds a proper class into the cardinals of a model of $\ZF$, however it is unclear that the class function $i\mapsto S_i$ definable internally to that model. Roguski's proof shows, instead, that given any set of cardinals, there is one incomparable to all of them. Using \autoref{Thm:ClassThm} we will show that such result can be extended so that $\DC_{<\kappa}$ is preserved for a fixed $\kappa$, and that we may replace $|S_i|\leq|S_j|$ by $|S_i|\leq^\ast|S_j|$.

Let $\fM$ be a countable transitive model of $\ZFC+\axiom{GCH}$, $\kappa$ a regular cardinal in $\fM$ and $(I,\preceq)$ a partially ordered class in $\fM$ such that every initial segment of $I$ is a set of $\fM$. Without loss of generality we may assume that there is a class in $\fM$ which well-orders $I$, for otherwise we can force such class without adding sets. Therefore we may assume that $I\subseteq\Ord^\fM$. In this section we shall prove the following theorem:

\begin{theorem}\label{Thm:ClassThm}
There exists a class-generic extension $\fM[G]$ with an intermediate model $\fN\subseteq\fM[G]$ such that $\fN\models\ZF+\DC_{<\kappa}$ in which $(I,\preceq)$ can be embedded into the cardinals of $\fN$ with the order $\leq$ or with the order $\leq^\ast$, such that every initial segment of this embedding if in $\fN$ and the embedding is definable in $\fM[G]$.
\end{theorem}

By embedding $(\power(I)\cap\fM,\subseteq)$ into the cardinals of the symmetric extension we will assure that $I$ has been embedded into it using the same argument as in the previous section. Note that if $I$ is actually a set in $\fM$ then \autoref{Thm:IncomparabilityOfChunks} proves the claim, so we may assume that $I$ is a proper class of $\fM$. We aim to mimic the previous proof therefore for every $i\in I$ we shall add generic subsets to a regular cardinal. In order to preserve $\DC_{<\kappa}$ we require the forcing to be $\kappa$-closed, so we will only add subsets to cardinals above $\kappa$. We will assume that $I$ is a class of regular cardinals and $\min I\geq\kappa$.

We define the forcing in $\fM$. For every $i\in I$ let $\PP_i=(P_i,\leq)$ be the forcing which adds $i$ subsets to $i$, namely $p\in P_i$ is a partial function from $i\times i$ to $2$ such that $|\dom p|<i$, and $p\leq q$ if and only if $q\subseteq p$. Let $\PP$ be the Easton support product $\prod_{i\in I}\PP_i$. We shall denote $\PP^{\leq i}$ the Easton support product of $\PP_j$ for $j\leq i$. This is a product of $\kappa$-closed forcings and therefore it is $\kappa$-closed, and we also point out that by assuming $\axiom{GCH}$ it does not change cofinalities.

The conditions in $\PP$ are functions such that $p(i)$ is a condition in $P_i$, and for all $i\in I$ we have $|\set{j\leq i}[p(j)\neq 1_{\PP_j}]|<i$. Alternatively we may think about the conditions as functions from $I\times\Ord\times\Ord$ to $\set{0,1}$ such that if $(i,\alpha,\beta)$ is in $\dom p$ then $\beta,\alpha<i$, and for every $i\in I$, $|\set{(j,\alpha,\beta)}[(j,\alpha,\beta)\in\dom p\land j\leq i]|<i$. We will identify $\PP^{\leq i}$ with those $p\in\PP$ such that $\dom p\subseteq (i+1)\times i\times i$.

If $G$ is a $\PP$-generic class over $\fM$ then in $\fM[G]$ it defines $i$ new subsets for every (regular) cardinal in $I$, and $\fM[G]$ is a model of $\ZFC$. Note that as before $\bigcup G=g$ is a class function $g\colon I\times\Ord\times\Ord\to 2$. We define the following sets from $G$ and give them canonical names:
\begin{itemize}
\item Let $i\in I$ and $\alpha<i$ then $r_{i,\alpha}=\set{\gamma<i}[g(i,\alpha,\gamma)=1]$ is given the name \[\dot r_{i,\alpha}=\set{(p,\check\gamma)}[p(i,\alpha,\gamma)=1\land p\in\PP^{\leq i}].\]
\item Let $i\in I$ we define $R_i = \set{r_{i,\alpha}}[\alpha<i]$ with the name \[\dot R_i=\set{\dot r_{i,\alpha}}[\alpha<i]^\bullet.\]
\item For a set $Q\subseteq I$ denote $D_Q=\bigcup_{i\in Q}R_i$. Of course $Q$ might be generic, and as before we do not give a name for $D_Q$.
\end{itemize}

We shall now proceed to define the symmetric extension. First we define $\cG$ to be a group of automorphisms of $\PP$, while this group will be a proper class each permutation will only move a set. We say that $\pi\in\cG$ if $\pi$ is a permutation of $I\times\Ord\times\Ord$ such that the following holds:
\begin{enumerate}
\item For all $i\in I$, if $(i,\alpha,\gamma)\in\dom\pi$ then $\alpha,\gamma<i$;
\item whenever $\pi(i,\alpha,\gamma)=(i',\alpha',\gamma')$ we have $i=i'$, $\alpha'<i$, $\gamma=\gamma'$;
\item $\DD_\pi=\set{(i,\alpha,\gamma)}[\pi(i,\alpha,\gamma)\neq(i,\alpha,\gamma)]$ is a set in $\fM$;
\item and for every $i\in I$, $|\set{(i,\alpha,\gamma)}[(i,\alpha,\gamma)\in \DD_\pi]|<i$.
\end{enumerate} 
  We define the action of $\cG$ on $\PP$ as before, \[\pi p(\pi(i,\alpha,\gamma))=p(i,\alpha,\gamma).\]

For each $i\in I$ we define $\cG_i=\set{\pi\in\cG}[\DD_\pi\subseteq (i+1)\times i\times i]$. Then $\cG_i$ is a group, $i\leq j$ implies $\cG_i\leq\cG_j$ and $\bigcup_{i\in I}\cG_i\simeq\cG$. We define $G_{\leq i}=G\cap\PP^{\leq i}$. We observe that $\fM^{\PP^{\leq i}}\subseteq\fM^{\PP^{\leq j}}$ whenever $i\leq j$, and that if $\dot x\in\fM^{\PP^{\leq i}}$ then $\dot x^{G_i}=\dot x^{G_j}$ as well.

Let $K_i=[(I\cap i^+)\times i\times i]^{<\kappa}$, and let $\cF_i$ be the $\kappa$-complete filter of subgroups of $\cG_i$ generated by $\fix(E)$ for $E\in K_i$, \[\cF_i=\set{H\leq\cG_i}[\exists E\in K_i: \fix(E)\leq H].\]
We have that for $i\leq j$, $\cF_i\subseteq\cF_j$. For every $i\in I$ let $\HS_i$ be $\HS_{\cF_i}\subseteq\fM^{\PP^{\leq i}}$. Let $\cF=\bigcup_{i\in I}\cF_i$, then $\HS=\HS_\cF=\bigcup_{i\in I}\HS_i$.

It is a standard way to define $\fM[G]$ as the union $\bigcup_{i\in I}\fM[G_i]$. For every $i\in I$ we define $\fN_i=(\HS_i)^{G_i}\subseteq\fM[G_i]$ to be a symmetric extension of $\fM$. Then for $i\leq j$ we have $\fN_i\subseteq\fN_j$, and every $\fN_i$ has the same ordinals (and initial ordinals) and satisfies $\ZF+\DC_{<\kappa}$.

Let $\fN=\bigcup_{i\in I}\fN_i$. We first observe the following equality holds, \[\fN=\bigcup_{i\in I}\fN_i=\bigcup_{i\in I}(\HS_i)^{G_i}=\bigcup_{i\in I}(\HS_i)^G=\left(\bigcup_{i\in I}\HS_i\right)^G=\HS^G.\]
Then $\fN$ is a model of $\ZF+\DC_{<\kappa}$. The model satisfies $\DC_{<\kappa}$ because every $\fN_i$ does. We shall see that $r_{i,\alpha}, R_i$ and $D_Q$ (for $Q\in\fN$) are in $\fN$, and that the class $\{D_Q\mid Q\in\power(I)\cap\fN\}$ is a class of $\fN$. Note that the equality above tells us that $x\in\fN$ if and only if there exists $\dot x\in\HS$ such that $\dot x^G=x$, and therefore there exists  $E\in[I\times \Ord\times\Ord]^{<\kappa}\cap\fM$ such that $\fix(E)\leq\sym(\dot x)$. As before we say that $E$ is a support of $\dot x$.

\begin{prop}
For all $i\in I,\alpha<i$ we have $r_{i,\alpha}\in\fN$, $R_i\in\fN$. Furthermore $i\mapsto R_i$ is definable in $\fM[G]$ and its initial segments are in $\fN$, and therefore whenever $Q\subseteq I$ is a set in $\fN$ then $D_Q$ is in $\fN$.
\end{prop}
\begin{proof}
It is immediate that $\set{(i,\alpha,0)}$ is a support of $\dot r_{i,\alpha}$ and that $\dot r_{i,\alpha}\in\HS_j$ for any $j>i$. It therefore follows that $\varnothing$ is a support of $\dot R_i$ which is also in $\HS_j$. From that it follows that $r_{i,\alpha}$ and $R_i$ are both in $\fN$ for any $i$ and $\alpha$.

Consider the class name $\dot F=\set{(\check i,\dot R_i)^\bullet}[i\in I]^\bullet$. For all $j\in I$ we have that the name $\dot F_j=\set{(\check i,\dot R_i)^\bullet}[i\leq j]^\bullet$ is a symmetric name in $\HS_j$, and $\sym_{\cG_j}(\dot F_j)=\cG_j$. Therefore $\dot F=\bigcup_{i\in I}\dot F_i$ is a symmetric class in $\HS$, and its interpretation $F=\dot F^G$ is a class of $\fM[G]$, whose initial segments are in $\fN$. From this follows that whenever $Q\subseteq I$ is a set in $\fN$ then it appears in some $\HS_j$ and therefore \[D_Q=\bigcup_{i\in Q}F(i)=\set{x}[\exists i\in Q:x\in F(i)],\] is in $\fN$ as promised.
\end{proof}

We observe that as in \autoref{Fact:Onto} every $R_i$ can be mapped onto $i$ with the map $r_{i,\alpha}\mapsto\min r_{i,\alpha}$. Obviously there is no $R_i$ that can be mapped onto $i^+$. Furthermore the proof that $\kappa\nleq|R_i|$ is the same as the proof of \autoref{Prop:ACfails}.

\begin{prop}\label{Prop:ComparabilityClass}
For every $Q,T$ subsets of $I$, $\fN\models Q\subseteq T\leftrightarrow|D_Q|\leq^\ast|D_T|$.
\end{prop}
\begin{proof}
If $Q\subseteq T$ then $D_T\subseteq D_Q$ and the result is trivial. Suppose that $\dot Q$ and $\dot T$ are names in $\HS$ for the sets $Q,T$ respectively. Assume towards contradiction that $p\forces\dot Q\nsubseteq\dot T\land\expr{\dot f\colon\dot D_T\to\dot D_Q\text{ is a surjection}}$ where $\dot f,\dot D_Q,\dot D_Q\in\HS$ and $\dot D_Q,\dot D_T$ are names for $D_Q$ and $D_T$ respectively.

The proof of \autoref{Thm:IncomparabilityOfChunks} applies here completely by noting that there is some $i\in I$ such that the entire proof is actually carried in $\PP^{\leq i}$.
\end{proof}

This concludes the proof of \autoref{Thm:ClassThm}, as taking $S_i=\set{j\preceq i}$ guarantees that $|S_i|\leq^\ast|S_j|\leftrightarrow|S_i|\leq|S_j|\leftrightarrow i\preceq j$.

\section{Extensions of the Theorem}

We draw two corollaries from the theorem and show the independence of two choice principles from $\DC_\kappa$ (for any $\kappa$). When a choice principle is not provable by $\DC_\kappa$, for any $\kappa$, it hints us that it may be equivalent to the axiom of choice, or that it is ``orthogonal'' to $\DC_\kappa$-like principles. 

\begin{theorem}\label{Thm:Antichains}
For every cardinal $\mu$ it is consistent with $\ZF+\DC_\mu$ that for every set of cardinals there is a cardinal incomparable to all of them, and for every ordinal $\alpha$ there is a decreasing sequence of cardinals of order type $\alpha^\ast$.
\end{theorem}
\begin{proof}
Let $\kappa>\mu$ and consider the model from Section~4 in which we embed class $I$ of regular cardinals above $\kappa$ with the discrete order into the cardinals of $\fN$, for better readability we identify the $I$ with its transitive collapse, $\Ord$. We have that $\fN\models\DC_{<\kappa}$ and therefore $\fN\models\DC_\mu$. For any $\alpha\in\Ord$ we have $R_\alpha$ as defined as in the proof of \autoref{Thm:ClassThm}, then $\set{|R_\alpha|}[\alpha<\lambda]$ is an antichain in both $\leq^\ast$ and in $\leq$, for every $\lambda\in\Ord$. 

Let $\alpha$ be an ordinal, and define for $\beta<\alpha$ the set $D_{\beta^\ast}=\bigcup\set{R_\gamma}[\beta\leq\gamma<\alpha]$. By \autoref{Prop:ComparabilityClass} we have that $\set{|D_{\beta^\ast}|}[\beta<\alpha]$ is a decreasing chain of cardinals in both $\leq$ and $\leq^\ast$.
\end{proof}
Note that it is impossible to find a decreasing sequence of cardinals of order type $\Ord^\ast$. Any set whose cardinality is the maximum of such sequence must have a proper class of different subsets, which is a contradiction to the power set axiom.

Finally we will show that the constructive set-theory oriented axiom known as $\WISC$ is independent from $\ZF$. This result was known due to van den Berg (see \cite{vandenBerg:2012})\footnote{Van den Berg names this axiom \textit{the Axiom of Multiple Choice}, an unfortunate name as it is already the name of a relatively known choice principle.}. The proof given by van den Berg assumes some very large cardinals, and we improve it by removing this additional assumption and by showing the compatibility of this failure with $\DC_\mu$ for arbitrary $\mu$.

The principle $\WISC$ (Weakly Initial Set Cover) can be formulated as follows: \textit{For every set $X$ there is a set $Y$, such that whenever $Z$ is a set and $f\colon Z\to X$ is a surjection then there is $q\colon Y\to Z$ such that $f\circ q$ is onto $X$}. This formulation is due to Fran\c{c}ois Dorais (see \cite{Roberts:2013} for more details).

We will now show that the model from \autoref{Thm:Antichains} satisfies $\lnot\WISC$. Recall that for all $\alpha$, $R_\alpha$ can be mapped onto $\kappa$. We will show that for $\kappa$ there is no $Y$ as in the requirement of $\WISC$.

\begin{theorem}
Let $\fN$ be the model from \autoref{Thm:Antichains}, then for every set $Y\in\fN$, there is $\alpha\in\Ord$ such that $R_\alpha$ can be mapped onto $\kappa$ by some function $h$, but every function $f\colon Y\to R_\alpha$ has range of cardinality $<\kappa$. Therefore there is no such $f$ for which $h\circ f$ is onto $\kappa$. In particular $\fN\models\lnot\WISC$. 
\end{theorem}
\begin{proof}
We will show that in $\fN$ for every $Y$ there is some $\alpha$ such that any $f\colon Y\to R_\alpha$ must satisfy $|\rng f|<\kappa$, and therefore it is impossible that any composition of $f$ with a function from $R_\alpha$ is onto $\kappa$.

Let $Y\in\fN$ be any set, and let $\alpha\in\Ord$ be such that for some $\beta<\alpha$ we have $\dot Y\in\HS_\beta$. This means that any condition which appears in $\dot Y$ appears in $\PP^{\leq\beta}$. Suppose that $p\forces\dot f\colon\dot Y\to\dot R_\alpha$, and $\dot f\in\HS$. If $p\forces|\dot\rng f|<\kappa$ then we are done, assume that this is not the case, and that $p\forces|\dot\rng f|\nless\kappa$.

Let $E\in[\Ord\times\Ord\times\Ord]^{<\kappa}$ be a support for $\dot f,\dot Y$ (recall that $\dot R_\alpha$ is supported by any set). Let $q\leq p$ be such that there is $\delta<\alpha$ such that for all $\gamma<\alpha$, $(\alpha,\delta,\gamma)\notin E$, and for some $\dot y$ we have that $q\forces\dot f(\dot y)=\dot r_{\alpha,\delta}$. We can now find $\tau\neq\delta$ such that $(\alpha,\tau,\gamma)\notin E\cup\dom q$ for any $\gamma<\alpha$. Let $\pi$ be the permutation in $\cG$ defined as follows:  $\pi(\alpha,\delta,\gamma)=(\alpha,\tau,\gamma)$; $\pi(\alpha,\tau,\gamma)=(\alpha,\delta,\gamma)$; and $\pi(x,y,z)=(x,y,z)$ otherwise.

As $\dot y$ is a name appearing in $\dot Y$, and thus $\dot y\in\HS_\beta$, we have that any condition in $\dot y$ appears in $\PP^{\leq\beta}$. This means that for any permutation in $\cG$ which does not move any condition in $\PP^{\leq\beta}$ will not move $\dot y$ either, in particular this is true for $\pi$ defined above.

We have that $\pi q\forces\dot f(\dot y)=\dot r_{\alpha,\tau}$, and as in the proof of \autoref{Thm:IncomparabilityOfChunks} we have that $q$ and $\pi q$ are compatible which is a contradiction, and the conclusion follows as wanted.
\end{proof}

Therefore for every $\mu$, $\WISC$ is unprovable from $\ZF+\DC_\mu$. This extends the results by Rathjen which establish the independence of a slightly stronger choice principle from $\ZF$ by a similar method as van den Berg (see \cite{Rathjen:2006}).

\section{Acknowledgements}

The author wishes to thank Uri Abraham and Matatyahu Rubin for many conversations which helped to shape this paper, and for their help in revising the manuscript. He also thanks Andr\'es Caicedo for introducing him to the problem of antichains of cardinals, and for additional suggestions. And to David M. Roberts for his help with the parts regarding $\WISC$. Final thanks goes to the referee of this paper for his helpful comments and corrections, and to the editor in his invaluable help in preparation of the final version.

\bibliographystyle{amsalpha}
\bibliography{bib}
\end{document}